\documentclass[12pt, a4paper, oneside]{article}
\usepackage[top=1in, bottom=1.1in, left=0.9in, right=0.8in]{geometry}
\usepackage[parfill]{parskip}
\usepackage{amsmath}
\usepackage{amsthm}
\usepackage{graphicx}
\usepackage{amssymb}
\usepackage{epstopdf}
\usepackage{setspace}
\usepackage{authblk}
\usepackage{enumerate}
\usepackage{lmodern}
\usepackage[hypcap]{caption}
\usepackage[english]{babel}
\usepackage{fancyhdr}
\addto\captionsenglish{}
\addto\captionsenglish{}
\addto\captionsenglish{}
\addto\captionsenglish{}
\addto\captionsenglish{}

\newlength\longest

\newcommand{\cM}{{\mathcal M}}

\newcommand{\N}{{\mathbb N}}
\newcommand{\R}{{\mathbb R}}

\newcommand{\cG}{{\mathcal G}}

\begin{document}

\newtheorem{theorem}{Theorem}[section]
\newtheorem{lemma}[theorem]{Lemma}
\newtheorem{corollary}[theorem]{Corollary}
\newtheorem{proposition}[theorem]{Proposition}
\newtheorem{conjecture}[theorem]{Conjecture}
\newtheorem{problem}[theorem]{Problem}
\newtheorem{claim}[theorem]{Claim}
\theoremstyle{definition}
\newtheorem{assumption}[theorem]{Assumption}
\newtheorem{remark}[theorem]{Remark}
\newtheorem{definition}[theorem]{Definition}
\newtheorem{example}[theorem]{Example}
\theoremstyle{remark}
\newtheorem{notation}{Notasi}
\renewcommand{\thenotation}{}

\title{Generalized H\"{o}lder's Inequality in Morrey Spaces}
\author{Ifronika${}^{1}$, Mochammad Idris${}^2$\footnote{\emph{Permanent Address}:
		Department of Mathematics, Universitas Lambung Mangkurat,
		Banjarbaru 70714, Indonesia}, Al Azhary Masta${}^3$\footnote{\emph{Permanent
        Address}: Department of Mathematics Education, Universitas Pendidikan Indonesia,
		Bandung 40154, Indonesia}, and Hendra Gunawan${}^{4}$}
\affil{
Analysis and Geometry Group,\\Faculty of Mathematics and Natural Sciences,\\
Bandung Institute of Technology,\\Bandung 40132, INDONESIA\\
\bigskip
E-mail addresses: ${}^{1}$ifronika@math.itb.ac.id, ${}^{2}$idemath@gmail.com,
${}^{3}$alazhari.masta@upi.edu, ${}^{4}$hgunawan@math.itb.ac.id}

\date{}

\maketitle

\begin{abstract}
 The aim of this paper is to present sufficient and necessary conditions for generalized
 H\"{o}lder's inequality in generalized Morrey spaces. We also obtain similar results in
 weak Morrey spaces and in generalized weak Morrey spaces. The sufficient and necessary
 conditions for the generalized  H\"{o}lder's inequality in these spaces are obtained
 through estimates for characteristic functions of balls in $\R^d$.

\bigskip


\medskip

\end{abstract}

\section{Introduction and Preliminaries}

Several authors have made important observations about H\"{o}lder's inequality
in the last three decades (see \cite{Avram, Cheung, Matkowski, Vasyunin}). Recently,
Masta \textit{et al.} \cite{Masta2} obtained sufficient and necessary conditions
for the generalized H\"{o}lder's inequality in Lebesgue spaces. In this paper,
we are interested in studying the generalized H\"{o}lder's inequality in Morrey
spaces and in generalized Morrey spaces. In particular, we shall prove sufficient
and necessary conditions for generalized H\"{o}lder's inequality in those spaces.
In addition, we also prove similar result in weak Morrey spaces and in generalized
weak Morrey spaces.

Let us first recall the definition of Morrey spaces. For $1 \le p \le q < \infty$,
the \textit{Morrey space} $ \cM^p_q(\R^d)$ is the set of all $p$-locally integrable
functions $f$ on $\R^d$ such that
\[
\|f\|_{\cM^p_q}:=\sup_{a\in\R^d, r>0}\ |B(a,r)|^{\frac{1}{q} - \frac{1}{p}}
\Bigl( \int_{B(a,r)} |f(y)|^p \ dy\Bigr)^{\frac{1}{p}} < \infty.
\]
Here, $B(a,r)$ denotes the open ball in $\mathbb{R}^d$ centered at $a$ with radius
$r>0$, and $|B(a,r)|$ denotes its Lebesgue measure. One might observe that
$\|\cdot\|_{\cM^p_q}$ defines a norm on $\cM^p_q(\R^d)$, and makes the space complete
\cite{SST}. Also note that if $q = p$, then $\cM^p_q(\R^d) = L^p(\R^d)$. Thus,
$\cM^p_q(\R^d)$ can be viewed as a generalization of the Lebesgue space $L^p(\R^d)$.

The following theorem presents sufficient and necessary conditions for H\"{o}lder's
inequality in Morrey spaces.

\bigskip

\begin{theorem}\label{theorem:1.1}
Let $1\leq p \leq q<\infty$, $1\leq p_1\leq q_1<\infty$, and $1\leq p_2\leq q_2<\infty$.
Then the following statements are equivalent:

{\parindent=0cm
{\rm (1)} $\frac{1}{p_1}+\frac{1}{p_2}\leq \frac{1}{p}$ and $\frac{1}{q_1}+\frac{1}{q_2}=
\frac{1}{q}$.
			
{\rm (2)} $\| f g \|_{\mathcal{M}^{p}_{q}} \leq \| f\|_{\mathcal{M}^{p_1}_{q_1}}
\|g\|_{\mathcal{M}^{p_2}_{q_2}}$ for every $f \in \mathcal{M}^{p_1}_{q_1}(\mathbb{R}^d)$
and $g \in \mathcal{M}^{p_2}_{ q_2}(\mathbb{R}^d)$.
\par}
\end{theorem}

\bigskip

Let us now move to the weak Morrey spaces. For $1\le p\le q<\infty$, the {\it weak
Morrey space} $w\mathcal{M}^p_q(\mathbb{R}^d)$ is the set of all measurable
functions $f$ on $\R^d$ for which $\|f\|_{w\mathcal{M}^p_q} < \infty$, where
\[
\|f\|_{w\mathcal{M}^p_q} := \sup_{a\in\R^d,\,r,\gamma > 0}
|B(a,r)|^{\frac{1}{q} - \frac{1}{p}}\,\gamma\left|\{x \in B(a,r)\,:\,|f(x)| > \gamma\}
\right|^{\frac{1}{p}}.
\]

Note that $\|\cdot\|_{w\mathcal{M}^p_q}$ defines a quasi-norm on $w\mathcal{M}^p_q(\mathbb{R}^d)$.
If $q = p$, then $w\cM^p_q(\R^d) = wL^p(\R^d)$. Here, $w\mathcal{M}^p_q(\mathbb{R}^d)$ can be viewed
as a generalization of the weak Lebesgue space $wL^{p}(\mathbb{R}^{d})$. The relation between
$w\mathcal{M}^p_q(\mathbb{R}^d)$ and $\mathcal{M}^p_q(\mathbb{R}^d)$ is shown in the following lemma.

\bigskip

\begin{lemma}\label{lemma:1.2}{\rm \cite{Gunawan}}
Let $1\le p\le q<\infty$. Then $\mathcal{M}^p_q(\mathbb{R}^d) \subseteq w\mathcal{M}^p_q(\mathbb{R}^d)$
with
\[
\|f\|_{w\mathcal{M}^p_q} \leq \|f\|_{\cM^p_q}
\]
for every $f \in \mathcal{M}^p_q(\mathbb{R}^d)$.
\end{lemma}

\bigskip

This lemma will be useful for us to study sufficient and necessary conditions
for generalized H\"{o}lder's inequality in weak Morrey spaces.

Next we present the definition of generalized Morrey spaces and generalized weak Morrey spaces.
For $1\leq p \leq q<\infty$, let $\mathcal{G}_p$ be the set of all functions $\phi: (0,\infty)
\rightarrow (0,\infty)$ such that $\phi$ is {\it almost decreasing} (i.e. there exists $C>0$
such that $\phi(r)\ge C\,\phi(s)$ for every $0<r<s<\infty$) and $r^{\frac{d}{p}}\phi(r)$ is
{\it almost increasing} (i.e. there exists $C>0$ such that $r^{\frac{d}{p}}\phi(r)\le C\,s^{\frac{d}{p}}
\phi(s)$ for every $0<r<s<\infty$). Note that if $\phi \in \mathcal{G}_p$, then $\phi$ satisfies the
{\it doubling condition}, that is, there exists $C > 0$ such that $\frac{1}{C} \le
\frac{\phi(r)}{\phi(s)} \le C$ whenever $ 1 \leq \frac{r}{s} \le 2$. For $ \phi \in \mathcal{G}_p$,
the \textit{generalized Morrey space} $\mathcal{M}^{p}_{\phi}(\mathbb{R}^d)$
is defined as the set of measurable functions $f$ on $\mathbb{R}^d$ for which
\[
\|f\|_{\cM^p_\phi}:=\sup_{a \in \mathbb{R}^d,r>0 } \ \frac{1}{\phi(r)}
\left( \frac{1}{|B(a,r)|} \int_{B(a,r)} |f(x)|^p\ dx  \right)^{\frac{1}{p}} <\infty.
\]
Note that $\mathcal{M}^{p}_{\phi}(\mathbb{R}^d) = \mathcal{M}^{p}_{q}(\mathbb{R}^d)$ for
$\phi(r):= r^{-\frac{d}{q}}$,  $ 1 \leq p \leq q < \infty$.
Meanwhile, for $ \phi \in \mathcal{G}_p$, the {\it generalized weak Morrey space}
$w\cM^p_{\phi}(\R^d)$ is defined to be the set of all measurable functions $f$ on
$\R^d$ such that
\[
\|f\|_{w\cM^p_{\phi}}:=\sup_{a\in \R^d,\, r,\gamma>0}
\frac{\gamma\left|\{x \in B(a,r)\,:\,|f(x)| > \gamma\}\right|^{\frac{1}{p}}}
{\phi(r) |B(a,r)|^{\frac{1}{p}}} < \infty.
\]
Here $\|\cdot\|_{w\mathcal{M}^p_{\phi}}$ is a quasi-norm on $w\mathcal{M}^p_{\phi}(\mathbb{R}^d)$.
Furthermore, $w\mathcal{M}^p_{\phi}(\mathbb{R}^d)=w\mathcal{M}^p_{q}(\mathbb{R}^d)$ for
$\phi(r) := r^{-\frac{d}{q}}$. The relation between the generalized Morrey spaces and
their weak type is given in the following lemma.

\bigskip

\begin{lemma}\label{lemma:1-3}
Let $1\le p<\infty$ and $\phi \in \cG_p$. Then $\cM^p_{\phi}(\mathbb{R}^d)\subseteq
w\cM^p_{\phi}(\mathbb{R}^d)$ with
\[
\|f\|_{w\cM^p_{\phi}} \le \|f\|_{\cM^p_{\phi}}
\]
for every $f\in \cM^p_{\phi}(\mathbb{R}^d)$.
\end{lemma}

\bigskip

In Section 2 we state our main results, and in Section 3 we present the proofs.

\section{Statement of The Results}

Our main results are presented in the following theorems. The first theorem is more general
than Theorem \ref{theorem:1.1}.

\bigskip

\begin{theorem}\label{theorem:2.3}
Let $m \ge 2$. If $1\le p \le q<\infty$ and $1\le p_i \le q_i < \infty$ for $i=1,\dots, m$,
then the following statements are equivalent:

{\parindent=0cm
{\rm (1)} $\sum \limits_{i=1}^m\frac{1}{p_i} \le \frac{1}{p}$ and
$\sum \limits_{i=1}^m\frac{1}{q_i}=\frac{1}{q}$.
			
{\rm (2)} $\left\| \prod\limits_{i=1}^m f_i \right\|_{\mathcal{M}^p_q} \le \prod\limits_{i=1}^m
\| f_i \|_{\mathcal{M}^{p_i}_{q_i}}$ for every $f_i\in \mathcal{M}_{q_i}^{p_i}(\mathbb{R}^d)$,
$i=1,\dots,m$.
				
\par}
\end{theorem}

\bigskip

\begin{theorem}\label{theorem:2.7}
Let $m\ge 2$. If $1\le p \le q<\infty$ and $1\le p_i \le q_i < \infty$ for $i=1,\dots, m$,
then the following statements are equivalent:

{\parindent=0cm
{\rm (1)} $\sum \limits_{i=1}^m\frac{1}{p_i} \leq \frac{1}{p}$ and $\sum \limits_{i=1}^m\frac{1}{q_i}=
\frac{1}{q}$.
		
{\rm (2)} $\left\| \prod\limits_{i=1}^m f_i \right\|_{w\mathcal{M}^p_q}\leq m \prod\limits_{i=1}^m
\| f_i \|_{w\mathcal{M}^{p_i}_{q_i}}$ for every $f_i\in w\mathcal{M}_{q_i}^{p_i}(\mathbb{R}^d)$,
$i=1,\dots,m$.
\par}
\end{theorem}

\medskip

For generalized Morrey spaces, we have the following theorems.

\bigskip

\begin{theorem}\label{theorem:3.1}
Let $m \ge 2$, $1 \le p, p_i < \infty$ with $\sum \limits_{i=1}^m\frac{1}{p_i}
\leq \frac{1}{p}$, $\phi \in \mathcal{G}_p$, and $\phi_i \in \mathcal{G}_{p_i}$
for $i = 1, \dots, m$.

{\parindent=0cm
{\rm (1)}
If $\prod\limits_{i=1}^m \phi_i(r) \le \phi(r)$ for every $r>0$, then
$\left\| \prod\limits_{i=1}^m f_i \right\|_{\mathcal{M}^p_\phi}
\leq \prod\limits_{i=1}^m \| f_i \|_{\mathcal{M}^{p_i}_{\phi_i}}$
for every $f_i\in \mathcal{M}_{\phi_i}^{p_i}(\mathbb{R}^d)$, $i=1,\dots,m$.
			
{\rm (2)}
If $\left\| \prod\limits_{i=1}^m f_i \right\|_{\mathcal{M}^p_\phi}
\leq \prod\limits_{i=1}^m \| f_i \|_{\mathcal{M}^{p_i}_{\phi_i}}$
for every $f_i\in \mathcal{M}_{\phi_i}^{p_i}(\mathbb{R}^d)$, $i=1,\dots,m$,
then there exists $C>0$ such that $\prod\limits_{i=1}^m \phi_i(r) \le C\,\phi(r)$
for every $r>0$.
\par}
\end{theorem}

\bigskip
\begin{theorem}\label{theorem:3.1a}
Let $m \ge 2$ and $1 \le p, p_i < \infty$ for $i=1,\dots,m$.
If $\phi \in\mathcal{G}_p$ and $\phi_i  \in \mathcal{G}_{p_i}$ such that 	
$\prod\limits_{i=1}^m \phi_i(r) = \phi(r)$ for every $r>0$ and there exists
$\epsilon>0$ such that $r^{\frac{\epsilon}{p_i}}\phi_i(r)$ are almost
decreasing for $i=1,\dots,m$, then the following statements are equivalent:

	{\parindent=0cm
	{\rm (1)} $\sum \limits_{i=1}^m\frac{1}{p_i} \le \frac{1}{p}$.
			
	{\rm (2)} $\left\| \prod\limits_{i=1}^m f_i \right\|_{\mathcal{M}^p_\phi}
	\leq \prod\limits_{i=1}^m \| f_i \|_{\mathcal{M}^{p_i}_{\phi_i}}$
	for every $f_i\in \mathcal{M}_{\phi_i}^{p_i}(\mathbb{R}^d)$, $i=1,\dots,m$.
	\par}
\end{theorem}

\noindent{\tt Remark}. In \cite{Sugano,Sugano1}, Sugano states that
$\left\| \prod\limits_{i=1}^m f_i \right\|_{\mathcal{M}^p_\phi}
\leq \prod\limits_{i=1}^m \| f_i \|_{\mathcal{M}^{p_i}_{\phi_i}}$ holds
for every $f_i\in \mathcal{M}_{\phi_i}^{p_i}(\mathbb{R}^d)$, $i=1,\dots,m$,
provided that $\sum_{i=1}^m \frac{1}{p_i} = \frac{1}{p }$ and
$\prod\limits_{i=1}^m \phi_i(r)= \phi (r)$. Theorems \ref{theorem:3.1}
and \ref{theorem:3.1a} may be viewed as counterparts of Sugano's results.

\medskip

Finally, for generalized weak Morrey spaces, the following theorems hold.

\medskip

\begin{theorem}\label{theorem:3.4}
Let $m \ge 2$, $1\le p, p_i < \infty$ with $\sum \limits_{i=1}^{m}\frac{1}{p_i}
\leq \frac{1}{p}$, $\phi \in \mathcal{G}_p$, and $\phi_i \in \mathcal{G}_{p_i}$
for $i = 1, \dots, m$.

{\parindent=0cm
{\rm (1)}
If $\prod\limits_{i=1}^m \phi_i(r) \le \phi(r)$ for every $r>0$, then
$\left\| \prod\limits_{i=1}^m f_i \right\|_{w\mathcal{M}^p_\phi}
\leq m \prod\limits_{i=1}^m \| f_i \|_{w\mathcal{M}^{p_i}_{\phi_i}}$
for every $f_i\in w\mathcal{M}_{\phi_i}^{p_i}(\mathbb{R}^d)$, $i=1,\dots,m$.
			
{\rm (2)}
If $\left\| \prod\limits_{i=1}^m f_i \right\|_{w\mathcal{M}^p_\phi}
\leq m \prod\limits_{i=1}^m \| f_i \|_{w\mathcal{M}^{p_i}_{\phi_i}}$
for every $f_i\in w\mathcal{M}_{\phi_i}^{p_i}(\mathbb{R}^d)$, $i=1,\dots,m$,
then there exists $C>0$ such that $\prod\limits_{i=1}^m \phi_i(r) \le C\,\phi(r)$
for every $r>0$.
\par}
\end{theorem}

\bigskip

\begin{theorem}\label{theorem:3.4a}
Let $m \ge 2$ and $ 1 \le p, p_i < \infty $ for $i=1,\dots,m$.
If $\phi \in\mathcal{G}_p$ and $\phi_i \in \mathcal{G}_{p_i}$ such that 	
$\prod\limits_{i=1}^m \phi_i(r) = \phi(r)$ for every $r>0$ and there exists
$\epsilon>0$ such that $r^{\frac{\epsilon}{p_i}}\phi_i(r)$ are almost
decreasing for $i=1,\dots,m$, then the following statements are equivalent:	

	{\parindent=0cm
	{\rm (1)} $\sum \limits_{i=1}^m\frac{1}{p_i} \le \frac{1}{p}$.
		
	{\rm (2)} $\left\| \prod\limits_{i=1}^m f_i \right\|_{w\mathcal{M}^p_\phi}
	\leq m \prod\limits_{i=1}^m \| f_i \|_{w\mathcal{M}^{p_i}_{\phi_i}}$
	for every $f_i\in w\mathcal{M}_{\phi_i}^{p_i}(\mathbb{R}^d)$, $i=1,\dots,m$.
	\par}
\end{theorem}

\section{Proof of Theorems}

Now we come to the proof of theorems in Section 2. Here, the letter $C$ denotes a constant that
may change from line to line. To prove our results, we shall use Lemma \ref{lemma:1.2}, Lemma
\ref{lemma:1-3}, and the following lemma.

\bigskip

\begin{lemma}\label{lemma:3.1}{\rm \cite{Eridani, EGU12, Gunawan}}
Let $1\le p<\infty$ and $\phi\in \cG_p$. Then there exists $C>0$ (depending on $\phi$) such that
\begin{equation}\label{eq1:140721}
\frac{1}{\phi(R)}\le \|\chi_{B(a_0,R)} \|_{w\cM^p_{\phi}} \le  \|\chi_{B(a_0,R)} \|_{\cM^p_{\phi}}
\le \frac{C}{\phi(R)}
\end{equation}
for every $a_0 \in \R^d$ and $R > 0$. In particular, we have
$$
R^{\frac{d}{q}}\le \|\chi_{B(a_0,R)} \|_{w\cM^p_q} \le \|\chi_{B(a_0,R)} \|_{\cM^p_q}\le
C\,R^{\frac{d}{q}}
$$
for every $a_0 \in \R^d$ and $R>0$.
\end{lemma}

\begin{proof}
This fact is proved in \cite{Eridani, EGU12, Gunawan}; we rewrite the proof here for convenience.
Let $B_0:=B(a_0,R) \subseteq \mathbb{R}^d$ where $a_0\in \R^d$ and $R>0$. If $r\le R$, then
$\phi(r) \ge C\,\phi(R)$, so that
\[
\frac{1}{\phi(r)} \Biggl( \frac{1}{|B(a,r)|} \int\limits_{B(a,r)} |\chi_{B_0}(x)|^p dx \Biggr)^{\frac{1}{p}}
\le \frac{C}{\phi(R)} \left( \frac{|B(a,r) \cap B_0|}{|B(a,r)|}\right)^{\frac{1}{p}}
\le \frac{C}{\phi(R)}
\]
for every $a\in \R^d$. If $r\ge R$, then $r^{\frac{d}{p}}\phi(r) \ge C\,R^{\frac{d}{p}}\phi(R),$ so that
\[
\frac{1}{\phi(r)} \Biggl( \frac{1}{|B(a,r)|} \int\limits_{B(a,r)} |\chi_{B_0}(x)|^p dx\Biggr)^{\frac{1}{p}}
\le \frac{C}{\phi(R)R^{\frac{d}{p}}} |B(a,r) \cap B_0|^{\frac{1}{p}}
\le \frac{C}{\phi(R)}
\]
for every $a\in \R^d$. Hence we conclude that $\| \chi_{B_0} \|_{\cM_{\phi}^p} \le \frac{C}{\phi(R)}$.
	
Next, by Lemma \ref{lemma:1-3}, we have
\[
\|\chi_{B_0} \|_{w\cM^p_{\phi}} \le \|\chi_{B_0} \|_{\cM^p_{\phi}}.
\]
Finally, by using the definition of $\|\cdot\|_{w\cM^p_{\phi}}$, we have
\begin{align*}
\|\chi_{B_0}\|_{w\cM^p_{\phi}} &\ge \frac{\gamma}{\phi(R)} \left( \frac{|\{x \in B_0:\,
|\chi_{B_0}(x)|>\gamma \}|}{|B_0|}\right)^{\frac1p}
=\frac{\gamma}{\phi(R)} \left( \frac{|B_0|}{|B_0|}\right)^{\frac1p}
=\frac{\gamma}{\phi(R)},
\end{align*}
for every $\gamma \in (0,1)$. Therefore $\|\chi_{B_0}\|_{w\cM^p_{\phi}} \ge \frac{1}{\phi(R)}$,
and the lemma is proved.
\end{proof}

\subsection{The proof of Theorem \ref{theorem:2.3}}

\begin{proof}
$(1) \Rightarrow (2)$ Let $\sum_{i=1}^{m}\frac{1}{p_i} \le \frac{1}{p}$ and
$\sum_{i=1}^{m} \frac{1}{q_i} =\frac{1}{q}$ hold. Put $\frac{1}{p^*}:=\sum_{i=1}^{m} \frac{1}{p_i}$.
Clearly $p^*\ge p$. Now take an arbitrary $B:=B(a,R)$ and $f_i \in \mathcal{M}^{p_i}_{q_i}(\R^d)$, where
$i=1,\dots,m$. By the generalized H\"{o}lder's inequality in Lebesgue spaces \cite{Cheung}, we have
$$
|B|^{\frac{1}{q}-\frac{1}{p}} \left( \int_{B} \prod\limits_{i=1}^m |f_i(x) |^{p} dx \right)^{\frac{1}{p}} \le
|B|^{\frac{1}{q}-\frac{1}{p^*}} \left( \int_{B} \prod\limits_{i=1}^m |f_i(x) |^{p^*} dx \right)^{\frac{1}{p^*}}
\le \prod\limits_{i=1}^m |B|^{\frac{1}{q_i}-\frac{1}{p_i}} \left( \int_{B}|f(x)|^{p_i} dx \right)^{\frac{1}{p_i}}.
$$
Taking the supremum over $B$, we obtain $\left\| \prod\limits_{i=1}^m f_i \right\|_{\mathcal{M}^{p}_{q}}
\le \prod\limits_{i=1}^m \| f_i \|_{\mathcal{M}^{p_i}_{q_i}}$.

\medskip

$(2) \Rightarrow (1)$ Suppose that $\left\| \prod\limits_{i=1}^m f_i \right\|_{\mathcal{M}^{p}_{q}} \le
\prod\limits_{i=1}^m \| f_i \|_{\mathcal{M}^{p_i}_{q_i}}$ for every $f_i \in \mathcal{M}^{p_i}_{q_i}(\R^d),
\ i=1,\dots,m$. Take an arbitrary $R>0$ and choose $f_i:=\chi_{B(0,R)}$ for $i=1,\dots,m$.
It follows from the hypothesis that
$$
\| \chi_{B(0,R)} \|_{\mathcal{M}^{p}_{q}} = \left\| \prod\limits_{i=1}^m f_i \right\|_{\mathcal{M}^p_q} \le
\prod_{i=1}^m \|f_i\|_{\mathcal{M}^{p_i}_{q_i}} = \prod\limits_{i=1}^m \|\chi_{B(0,R)}\|_{\mathcal{M}^{p_i}_{q_i}}.
$$
Hence, by Lemma \ref{lemma:3.1}, we have $R^{\frac{d}{q } - \sum_{i=1}^m \frac{d}{q_i}} \le C$.
Since $R > 0$ is arbitrary, we conclude that $\sum_{i=1}^m \frac{1}{q_i} = \frac{1}{q}.$

Next, choose $0 < \epsilon < \min \{\frac{dp_1}{q_1},\dots,\frac{dp_m}{q_m}\}$.
For arbitrary $K\in \N$, we define $g_{\epsilon,K}(x):=\chi_{\{0 \le |x| <1\}}(x)+\sum_{j=1}^{K}
\chi_{\{j \le |x| \le j+j^{-\epsilon}\}}(x)$. For $i=1,\dots,m$, we define
$f_i:=g_{\epsilon,K}$. Note that $\prod_{i=1}^m f_i=g_{\epsilon,K}$ and so
$\bigl|\prod_{i=1}^m f_i\bigr|^p = g_{\epsilon,K}$. Hence, we obtain
\begin{align*}
\left\| \prod\limits_{i=1}^m f_i\right\|_{\mathcal{M}^{p}_{q }}
&= \sup_{a\in\R^d,r>0 }\ |B(a,r) |^{\frac{1}{q} - \frac{1}{p}}
\left( \int_{B(a,r)} \bigl|\prod\limits_{i=1}^m f_i(x)\bigr|^p\,dx\right)^{\frac{1}{p}}\\
&\ge |B(0,K+K^{-\epsilon})|^{\frac{1}{q }-\frac{1}{p }}
\left(\int_{B(0,K+K^{-\epsilon})} g_{\epsilon,K}(x)\ dx\right)^{\frac{1}{p} }
\nonumber
\\
&= C(K+K^{-\epsilon})^{\frac{d}{q}-\frac{d}{p}}\left(|B(0,1)|+
\sum_{j=1}^{K} \int_{j\leq |x|\leq j+j^{-\epsilon}} dx \right)^{\frac{1}{p}}\\
&= C(K+K^{-\epsilon})^{\frac{d}{q}-\frac{d}{p}}\left(|B(0,1)|+
\sum_{j=1}^{K} (|B(0,j+j^{-\epsilon})|-|B(0,j)|) \right)^{\frac{1}{p}}\\
&\ge C(K+K^{-\epsilon})^{\frac{d}{q}-\frac{d}{p}}\left(\sum_{j=1}^{K}
\bigl[(j+j^{-\epsilon})^d-j^d\bigr] \right)^{\frac{1}{p}}\\
&\ge C (K+K^{-\epsilon})^{\frac{d}{q}-\frac{d}{p}} (K+K^{-\epsilon})^{\frac{d}{p }-\frac{\epsilon}{p}}\\
&= C(K+K^{-\epsilon})^{\frac{d}{q}-\frac{\epsilon}{p}}.
\end{align*}
Meanwhile, for each $i = 1, \dots, m$, we claim that
\[
\sup\limits_{a\in\R^d, r>0} |B(a,r)|^{\frac{1}{q_i} - \frac{1}{p_i}}
\Bigl({\displaystyle \int_{B(a,r)} |f_i(x)|^{p_i} \,dx}\Bigr)^{\frac{1}{p_i}}
=|B(0,L)|^{\frac{1}{q_i} - \frac{1}{p_i}}\Bigl({\displaystyle \int_{B(0,L)}
|f_i(x)|^{p_i} \,dx}\Bigr)^{\frac{1}{p_i}}
\]
for $2< L\leq K+K^{-\epsilon}$. To see this, note that $f_i=g_{\epsilon,K}$ is symmetrical about $0$
and has most mass around $0$, and so for each $a\in\R^d$ and $r>0$, we have
\[
|B(a,r)|^{\frac{1}{q_i} - \frac{1}{p_i}}
\Bigl({\displaystyle \int_{B(a,r)} |f_i(x)|^{p_i} \,dx}\Bigr)^{\frac{1}{p_i}} \le |B(0,r)|^{\frac{1}{q_i}
-\frac{1}{p_i}}\Bigl({\displaystyle \int_{B(0,r)} |f_i(x)|^{p_i} \,dx}\Bigr)^{\frac{1}{p_i}}.
\]
Now, as a function of $r$ only, the value of the last expression on the right hand side gets larger and
larger as $r$ grows from $0$ to $2+2^{-\epsilon}$ but decreases for $r>K+K^{-\epsilon}$. This verifies
our claim about the supremum. 

With such a value of $L$, let $L_1 :=\lfloor L \rfloor$ and $L_2:= \lceil L \rceil$. Clearly
$L_1\ge\frac{1}{2}L_2$. Since $\frac{1}{q_i} - \frac{1}{p_i} \le 0$ for $i=1, \dots, m$, we have
\begin{align*}
\|f_i \|_{\mathcal{M}^{p_i }_{q_i }} &= \sup_{a\in\R^d,r>0 }\ |B(a,r)|^{\frac{1}{q_i} - \frac{1}{p_i}}
\left( \int_{B(a,r)} |f_i(x)|^{p_i}\ dx\right)^{\frac{1}{p_i}}\\
&= |B(0,L)|^{\frac{1}{q_i} - \frac{1}{p_i}} \left( \int_{B(0,L)} |f_i(x)|^{p_i}\ dx\right)^{\frac{1}{p_i}}\\
&\leq C\,L_1^{\frac{d}{q_i}-\frac{d}{p_i}} \Biggl(|B(0,1)|+\sum_{j=1}^{L_2} \bigl[(j+j^{-\epsilon})^d-j^d\bigr]
\Biggr)^{\frac{1}{p_i}}\\
&\leq C\,L_1^{\frac{d}{q_i}-\frac{d}{p_i}} \Biggl(|B(0,1)|+\sum_{j=1}^{L_2} j^{d-1-\epsilon}\Biggr)^{\frac{1}{p_i}}\\
&\leq C\,L_1^{\frac{d}{q_i}-\frac{d}{p_i}} L_2^{\frac{d}{p_i}-\frac{\epsilon}{p_i}}\\
&\leq C\,L_2^{\frac{d}{q_i}-\frac{\epsilon}{p_i}}.
\end{align*}
Moreover, since $L_2 \le K+1 \le 2(K+K^{-\epsilon})$, we obtain
\[
\| f_i \|_{\mathcal{M}^{p_i}_{q_i}} \le C (K+K^{-\epsilon})^{\frac{d}{q_i}-\frac{\epsilon}{p_i}}
\]
for $i=1,\dots,m$.

Knowing that $\sum_{i=1}^m \frac{d}{q_i} = \frac{d}{q} $ and $\left\| \prod\limits_{i=1}^m
f_i \right\|_{\mathcal{M}^p_q} \le \prod\limits_{i=1}^m \| f_i \|_{\mathcal{M}^{p_i}_{q_i}}$,
we conclude from the two inequalities above that
\[
(K+K^{-\epsilon})^{-\frac{\epsilon}{p} + \sum_{i=1}^m \frac{\epsilon}{p_i}}\le C
\]
for every $K\in \mathbb{N}$. Therefore $\sum_{i=1}^m \frac{1}{p_i} \le \frac{1}{p}$, as desired.
\end{proof}

\noindent{\tt Remark}. For $m=2$, we obtain the proof of Theorem \ref{theorem:1.1}.

\subsection{The proof of Theorem \ref{theorem:2.7}}

\begin{proof}
If (1) holds, then by  similar arguments as in \cite{Indra} we can prove that (1) implies (2).
It thus remains to prove that (2) implies (1). To do so, take an arbitrary $R>0$ and let
$f_i := \chi_{B(0,R)}$ for $i=1,\dots,m$. By the hypothesis, we then have
$$
\| \chi_{B(0,R)} \|_{w\mathcal{M}^{p}_{q}} = \left\| \prod\limits_{i=1}^m f_i \right\|_{w\mathcal{M}^p_q}
\le m \prod\limits_{i=1}^m \|f_i\|_{w\mathcal{M}^{p_i}_{q_i}} = m \prod\limits_{i=1}^m
\| \chi_{B(0,R)} \|_{w\mathcal{M}^{p_i}_{q_i}}.
$$
Hence $R^{\frac{d}{q }-\sum_{i=1}^m \frac{d}{q_i}}\le C$.
Since this holds for every $R > 0$, it follows that $\sum_{i=1}^m \frac{1}{q_i} = \frac{1}{q}.$

Next, let $0 < \epsilon < \min \{\frac{dp_1}{q_1},\dots,\frac{dp_m}{q_m}\}$
and, for arbitrary $K\in \N$, define $g_{\epsilon,K}(x) := \chi_{\{0 \le |x| < 1\}}(x)+
\sum_{j=1}^K \chi_{\{j \le |x| \le j+j^{-\epsilon}\}}(x)$.
For $i=1,\dots,m$, let $f_i := g_{\epsilon,K}$. We observe that
\begin{align*}
\left\| \prod\limits_{i=1}^m f_i \right\|_{w\mathcal{M}^{p}_{q}}
&\ge \frac{1}{2}|B(0,K+K^{-\epsilon})|^{\frac{1}{q }-\frac{1}{p}}
\,\left|\bigl\{x \in B(a,r)\,:\,|f_i(x)| > \frac{1}{2} \bigr\}\right|^{\frac{1}{p}}
\nonumber\\
&\ge C(K+K^{-\epsilon})^{\frac{d}{q}-\frac{d}{p}}(K+K^{-\epsilon})^{\frac{d}{p}-\frac{\epsilon}{p}}
= C(K+K^{-\epsilon})^{\frac{d}{q}-\frac{\epsilon}{p}}.
\end{align*}
Meanwhile, by Lemma \ref{lemma:1.2} and the Morrey-norm estimate for $f_i$, we obtain
$
\| f_i \|_{w\mathcal{M}^{p_i}_{q_i}}\le \|f_i\|_{\mathcal{M}^{p_i}_{q_i}} \le
C(K+K^{-\epsilon})^{\frac{d}{q_i}-\frac{\epsilon}{p_i}}
$
for $i = 1, \dots, m$. Since $\sum_{i=1}^m \frac{1}{q_i}=\frac{1}{q}$ and $\left\| \prod\limits_{i=1}^m
f_i \right\|_{w\mathcal{M}^p_q} \le m \prod\limits_{i=1}^m \| f_i \|_{w\mathcal{M}^{p_i}_{q_i}}$, we have
\[
(K+K^{-\epsilon})^{-\frac{\epsilon}{p}+\sum_{i=1}^m \frac{\epsilon}{p_i}}\le C.
\]
Since it holds for every $K \in \mathbb{N}$, we must have $\sum_{i=1}^m \frac{1}{p_i}\leq \frac{1}{p}$.
\end{proof}

\subsection{The proof of Theorem \ref{theorem:3.1}}

\begin{proof}
(1) Suppose that $\prod\limits_{i=1}^m \phi_i(r) \le \phi(r)$ for every $r>0$. Take an arbitrary
$B:=B(a,R) \subseteq \mathbb{R}^d$ and $f_i \in \mathcal{M}^{p_i}_{\phi_i}(\mathbb{R}^d)$,
where $i=1,\dots,m$. Putting $\frac{1}{p^*} := \sum_{i=1}^m \frac{1}{p_i}$, it follows from
the generalized H\"{o}lder's inequality in Lebesgue spaces that
\begin{align*}
\frac{1}{\phi(R)}\left(\frac{1}{|B|} \int\limits_{B}\prod\limits_{i=1}^m |f_i(x)|^{p}dx \right)^{\frac{1}{p}} &\le
\frac{1}{\phi(R)}\left(\frac{1}{|B|} \int\limits_{B}\prod\limits_{i=1}^m |f_i(x)|^{p^*}dx \right)^{\frac{1}{p^*}}\\
&\le \prod\limits_{i=1}^m \frac{1}{\phi_i(R)} \left(\frac{1}{|B|} \int\limits_{B}|f_i(x)|^{p_i}dx \right)^{\frac{1}{p_i}}.
\end{align*}
We can now take the supremum over $B$ to obtain $\left\| \prod\limits_{i=1}^m f_i \right\|_{\mathcal{M}^{p}_{\phi}}
\le \prod\limits_{i=1}^m \|f_i\|_{\mathcal{M}^{p_i}_{\phi_i}}$.
	
\bigskip

(2) Suppose that $\left\| \prod\limits_{i=1}^m f_i \right\|_{\mathcal{M}^{p}_{\phi}}
\le \prod\limits_{i=1}^m \| f_i \|_{\mathcal{M}^{p_i}_{\phi_i}}$ for every
$f_i \in \mathcal{M}^{p_i}_{\phi_i}(\mathbb{R}^d)$, where $i=1,\dots,m$.
Take an arbitrary $R>0$ and define $f_i := \chi_{B(0,R)}$ for $i=1,\dots,m$.
Then there exists $C>0$ (independent of $R$) such that
$$
\frac{1}{\phi(R)} \le \| \chi_{B(0,R)} \|_{\mathcal{M}^{p}_{\phi}} \le \prod\limits_{i=1}^m
\| \chi_{B(0,R)} \|_{\mathcal{M}^{p_i}_{\phi_i}} \le \prod\limits_{i=1}^m \frac{C}{\phi_i(R)}.
$$
Thus $\prod\limits_{i=1}^m \phi_i(R) \le C\,\phi(R)$, as desired.
\end{proof}

\subsection{The proof of Theorem \ref{theorem:3.1a}}

\begin{proof}
	$(1) \Rightarrow (2)$ Suppose that $\sum_{i=1}^m \frac{1}{p_i} \leq \frac{1}{p }$.
    As before, one may easily observe that
	$\left\| \prod\limits_{i=1}^m f_i \right\|_{\mathcal{M}^{p}_{\phi}}
	\leq \prod\limits_{i=1}^m \|f_i\|_{\mathcal{M}^{p_i}_{\phi_i}}$ for every $f_i \in
    \mathcal{M}^{p_i}_{\phi_i},\ i=1,\dots,m$.
	
	\bigskip
	
	$(2) \Rightarrow (1)$ Let $\left\| \prod\limits_{i=1}^m f_i \right\|_{\mathcal{M}^{p}_{\phi}}
    \le \prod\limits_{i=1}^m \| f_i\|_{\mathcal{M}^{p_i}_{\phi_i}}$ for every
	$f_i \in \mathcal{M}^{p_i}_{\phi_i}(\mathbb{R}^d)$, where $i=1,\dots,m$.
	For arbitrary $K\in \N$, we define $g_{\epsilon,K}(x):=\chi_{\{ 0\le |x|<1 \}}(x)+
    \sum_{j=1}^{K} \chi_{ \{ j\le |x|\le j+j^{-\epsilon} \}}(x)$. For $i=1,\dots,m$, let
	$f_i:=g_{\epsilon,K}$. It is easy to check that
	\begin{align*}
	\left\| \prod\limits_{i=1}^m f_i \right\|_{\mathcal{M}^{p}_{\phi}}
	&\ge \frac{1}{\phi\left( K+K^{-\epsilon}\right)}\Biggl( \frac{1}{|B(0,K+K^{-\epsilon})|}
    \int\limits_{B(0,K+ K^{-\epsilon})} g_{\epsilon,K}(x)\ dx\Biggr)^{\frac{1}{p}}\\
	\nonumber
	&\ge \frac{C}{(K+K^{-\epsilon})^{\frac{\epsilon}{p}}\phi(K+K^{-\epsilon})}.
	\end{align*}
	Meanwhile, for $i = 1,\dots, m$, by using similar arguments as in the proof of Theorem
    \ref{theorem:2.3} one may observe that for $2<L\le K+K^{-\epsilon}$,
    \[
	\| f_i \|_{\mathcal{M}^{p_i }_{\phi_i}} \le \frac{1}{\phi_i(L)}
	\Biggl( \frac{1}{|B(0,L)|} \int\limits_{B(0,L)} g_{\epsilon,K}(x)\ dx \Biggr)^{\frac{1}{p_i}}
	\le \frac{C}{(K+K^{-\epsilon})^{\frac{\epsilon}{p_i}}\phi_i(K+K^{-\epsilon})}.
    \]
Because $\left\| \prod\limits_{i=1}^m f_i \right\|_{\mathcal{M}^p_\phi} \le \prod\limits_{i=1}^m
\| f_i \|_{\mathcal{M}^{p_i}_{\phi_i}}$ and $\prod\limits_{i=1}^m \phi_i(r) =\phi(r)$ for every
$r>0$, we get
\[
(K+K^{-\epsilon})^{-\frac{\epsilon}{p} + \sum_{i=1}^m \frac{\epsilon}{p_i}} \le C,
\]
which holds for arbitrary $K \in \mathbb{N}$. Consequently, $\sum_{i=1}^m \frac{1}{p_i} \le \frac{1}{p}$,
as desired.
\end{proof}

\subsection{The proof of Theorem \ref{theorem:3.4}}

\begin{proof}
(1) Suppose that $\prod\limits_{i=1}^m \phi_i(r) \le \phi(r)$ for every $r>0$.
Let $f_i \in w\mathcal{M}^{p_i}_{\phi_i}(\mathbb{R}^d),$ where $i=1,\dots,m$.
For an arbitrary $B:=B(a,R) \subseteq \mathbb{R}^d$ and $\gamma > 0$, let
\[
A(B,\gamma) := \left[ \frac{\gamma^p\left|  \left\{x \in B\,:\,
\prod\limits_{i=1}^m \bigl|\frac{f_i(x)}{\| f_i \|_{w\mathcal{M}^{p_i}_{\phi_i}}}\bigr|> \gamma \right\}
\right| }{\phi^{p}(R)|B|} \right]^{\frac{1}{p}}.
\]
Putting $\frac{1}{p^*}:=\sum_{i=1}^m \frac{1}{p_i}$, we observe that
\begin{align*}
A(B,\gamma) & \le \left[ \frac{\gamma^{p^*}\left| \left\{x \in B\,:\,\prod\limits_{i=1}^m
    \bigl| \frac{f_i(x)}{\| f_i \|_{w\mathcal{M}^{p_i}_{\phi_i}}} \bigr|> \gamma \right\} \right|}
    {|B|\bigl(\prod\limits_{i=1}^m\phi_i(R)\bigr)^{p^*}}\right]^{\frac{1}{p^*}}
    & = \left[ \frac{\gamma_0^{p^*}\left| \left\{x \in B\,:\,\prod\limits_{i=1}^m \bigl|\frac{f_i(x)}{\phi_i(R)
    \| f_i \|_{w\mathcal{M}^{p_i}_{\phi_i}}}\bigr|> \gamma_0 \right\} \right|}{|B|} \right]^{\frac{1}{p^*}},
\end{align*}
where $\gamma_0 := \frac{\gamma}{\prod\limits_{i=1}^m\phi_i(R)}$.
Furthermore, by using Young's inequality for products, we have
\begin{align*}
A(B,\gamma) & \le \left[ \frac{\gamma_0^{p^*} \left| \Bigl\{x\in B\,:\,\prod\limits_{i=1}^m \bigl|\frac{f_i(x)}
    {\phi_i(R)\| f_i \|_{w\mathcal{M}^{p_i}_{\phi_i}}}\bigr|> \gamma_0 \Bigr\} \right|}{|B|} \right]^{\frac{1}{p^*}}\\
    & \le \left[ \frac{\gamma_0^{p^*}\left| \left\{x \in B\,:\,\sum\limits_{i=1}^{m} \frac{p^*}{p_i}
    \bigl|\frac{f_i(x)}{\phi_i(R)\| f_i \|_{w\mathcal{M}^{p_i}_{\phi_i}}}\bigr|^{\frac{p_i}{p^*}} > \gamma_0
    \right\} \right|}{|B|}\right]^{\frac{1}{p^*}}\\
    & \le \left[ \sum_{i=1}^{m} \frac{\gamma_0^{p^*}\left| \left\{x \in B\,:\,\frac{p^*}{p_i} \bigl|\frac{f_i(x)}
    {\phi_i(R)\| f_i \|_{w\mathcal{M}^{p_i}_{\phi_i}}}\bigr|^{\frac{p_i}{p^*}} > \frac{\gamma_0}{m} \right\}
    \right|}{|B|} \right]^{\frac{1}{p^*}}.
\end{align*}
Notice that $\frac{p^*}{p_i}\Bigl| \frac{f_i(x)}{\phi_i(R)\|f_i\|_{w\mathcal{M}^{p_i}_{\phi_i}}}
\Bigr|^{\frac{p_i}{p^*}}>\frac{\gamma_0}{m}$ is equivalent to $|f_i(x)|>\bigl( \frac{\gamma_0 p_i}{mp^*} \bigr)^{p^*/p_i}
\phi_i(R)\|f_i\|_{w\mathcal{M}^{p_i}_{\phi_i}} =: \gamma_i$.
Hence we obtain
\begin{align*}
A(B,\gamma) & \le \left[ \sum_{i=1}^{m} \Bigl( \frac{\gamma_i(mp^*)^{p^*/p_i}}{\phi_i(R)
        p_i^{p^*/p_i}\| f_i \|_{w\mathcal{M}^{p_i}_{\phi_i}}} \Bigr)^{p_i}\, \frac{\left|
        \left\{ x \in B\,:\,|f_i(x)| > \gamma_i \right\} \right|}{|B|} \right]^{\frac{1}{p^*}}\\
        &= m \left[ \sum_{i=1}^m \Bigl(\frac{p^*}{p_i}\Bigr)^{p^*} \frac{\gamma_i^{p_i}|\{x\in B\,:\,|f_i(x)|>
        \gamma_i\}|}{\phi_i(R)^{p_i}|B|\|f_i\|_{w\mathcal{M}^{p_i}_{\phi_i}}^{p_i}} \right]^{\frac{1}{p^*}}\\
        &\le m \left[ \sum_{i=1}^m \Bigl(\frac{p^*}{p_i}\Bigr)^{p^*} \right]^{\frac{1}{p^*}}
        \le m \left[ \sum_{i=1}^m \frac{p^*}{p_i} \right]^{\frac{1}{p^*}}= m,
\end{align*}
because $1\le p^*\le p_i$ for $i = 1, \dots, m$.
We then take the supremum of $A(B,\gamma)$ over $B=B(a,R)$ and $\gamma > 0$ to obtain
$\left\| \prod\limits_{i=1}^{m} f_i \right\|_{w\mathcal{M}^{p}_{\phi}}
\leq m \prod\limits_{i=1}^{m} \| f_i \|_{w\mathcal{M}^{p_i}_{\phi_i}}$.

\bigskip

(2) Let $\left\| \prod\limits_{i=1}^m f_i \right\|_{w\mathcal{M}^{p}_{\phi}} \le
m \prod\limits_{i=1}^m \left\Vert f_i\right\Vert_{w\mathcal{M}^{p_i}_{\phi_i}}$
for every $f_i \in w\mathcal{M}^{p_i}_{\phi_i}(\mathbb{R}^d)$, $i=1,\dots,m$.
Take an arbitrary $R>0$ and define $f_i:=\chi_{B(0,R)}$ for $i=1,\dots,m$. By the hypothesis,
we have
\[
\| \chi_{B(0,R)} \|_{w\mathcal{M}^{p}_{\phi}} \le m \prod\limits_{i=1}^m
\| \chi_{B(0,R)} \|_{w\mathcal{M}^{p_i}_{\phi_i}}.
\]
It thus follows from Lemma \ref{lemma:3.1} that there exists $C>0$ (independent of $R$)
such that $\prod\limits_{i=1}^m \phi_i(R) \le C\,\phi(R)$.
\end{proof}

\subsection{The proof of Theorem \ref{theorem:3.4a}}

\begin{proof}
	$(1) \Rightarrow (2)$ Suppose that $\sum_{i=1}^m \frac{1}{p_i} \le \frac{1}{p}$.
    As before, we obtain $\left\| \prod\limits_{i=1}^{m} f_i \right\|_{w\mathcal{M}^{p}_{\phi}}
	\le m \prod\limits_{i=1}^{m} \| f_i \|_{w\mathcal{M}^{p_i}_{\phi_i}}$ for every
    $f_i \in w\mathcal{M}^{p_i}_{\phi_i},\ i=1,\dots,m$.
	
	\bigskip
	
	$(2) \Rightarrow (1)$ Suppose that $\left\| \prod\limits_{i=1}^m f_i \right\|_{w\mathcal{M}^{p}_{\phi}}
    \le m \prod\limits_{i=1}^m \| f_i \|_{w\mathcal{M}^{p_i}_{\phi_i}}$ for every $f_i \in
	w\mathcal{M}^{p_i}_{\phi_i}(\mathbb{R}^d)$, where $i=1,\dots,m$.
	For arbitrary $K\in\N$, define $g_{\epsilon,K}(x):=\chi_{\{ 0\le |x|<1 \}}(x)+\sum_{j=1}^K
    \chi_{ \{ j\le |x|\le j+j^{-\epsilon} \}}(x)$, and for $i=1,\dots,m$ put
	$f_i:=g_{\epsilon,K}$.
    By using the same arguments as in the proof of Theorem \ref{theorem:2.7}, we have
    \[
    \left\|\prod\limits_{i=1}^m f_i\right\|_{w\mathcal{M}^{p}_{\phi}}
    \ge
    \frac{C}{(K+K^{-\epsilon})^{\frac{\epsilon}{p }}\phi(K+K^{-\epsilon})}.
    \]
    Next, using Lemma \ref{lemma:1-3} and the generalized Morrey-norm estimate for $f_i$,
    we have
    \[
    \|f_i \|_{w\mathcal{M}^{p_i }_{\phi_i }} \le \| f_i \|_{\mathcal{M}^{p_i }_{\phi_i }}
	\le \frac{C}{(K+K^{-\epsilon})^{\frac{\epsilon}{p_i }}\phi_i(K+K^{-\epsilon})},
    \]
    for $i = 1, \dots, m.$
	Since $\left\| \prod\limits_{i=1}^m f_i \right\|_{w\mathcal{M}^p_\phi} \le m \prod\limits_{i=1}^m
	\| f_i \|_{w\mathcal{M}^{p_i}_{\phi_i}}$ and $\prod\limits_{i=1}^m \phi_i(r) = \phi(r)$
    for every $r>0$, it follows that
	\[
	(K+K^{-\epsilon})^{-\frac{\epsilon}{p }+\sum_{i=1}^m \frac{\epsilon}{p_i}} \le C.
	\]
	We therefore conclude that $\sum_{i=1}^m \frac{1}{p_i}\leq \frac{1}{p}$.
\end{proof}

\section{Concluding Remarks}
We have shown sufficient and necessary conditions for generalized H\"{o}lder's inequality in
several spaces, namely Morrey spaces and their weak type versions.
From Theorems \ref{theorem:2.3} and \ref{theorem:2.7}, we see that both H\"older's inequality
in Morrey spaces and in weak Morrey spaces are equivalent to the same condition, namely
$\sum_{i=1}^m \frac{1}{p_i} \le \frac{1}{p}$ and $\sum_{i=1}^m \frac{1}{q_i} = \frac{1}{q}$.
Accordingly, we have the following corollary.

\bigskip

\begin{corollary}
For $m \ge 2$, the following statements are equivalent:

{\parindent=0cm
{\rm (1)} $\left\| \prod\limits_{i=1}^{m} f_{i} \right\|_{\mathcal{M}^p_q} \le
\prod\limits_{i=1}^{m} \|f_{i}\|_{\mathcal{M}^{p_i}_{q_i}}$ for every
$f_i \in \mathcal{M}^{p_i}_{q_i} (\R^d)$, where $i = 1,\dots,m$.
	
{\rm (2)} $\left\| \prod\limits_{i=1}^{m} f_{i} \right\|_{w\mathcal{M}^{p}_{q}}
\le m \prod\limits_{i=1}^{m} \|f_{i} \|_{w\mathcal{M}^{p_i}_{q_i}}$ for every
$f_i \in w\mathcal{M}^{p_i}_{q_i} (\R^d)$, where $i = 1,\dots,m$.
\par}
\end{corollary}

\medskip

Similarly, from Theorems 2.4 and 2.6, we have the following corollary about H\"older's
inequality in generalized Morrey spaces and in generalized weak Morrey spaces.

\bigskip

\begin{corollary}
Let $m \ge 2$ and $1 \le p, p_i < \infty $ for $i=1,\dots,m$.
If $\phi \in\mathcal{G}_p$ and $\phi_i  \in \mathcal{G}_{p_i}$ such that 	
$\prod\limits_{i=1}^m \phi_i(r) = \phi(r)$ for every $r>0$ and there exists
$\epsilon>0$ such that $r^{\frac{\epsilon}{p_i}}\phi_i(r)$ are almost
decreasing for $i=1,\dots,m$, then the following statements are equivalent:

	{\parindent=0cm
	{\rm (1)} $\left\| \prod\limits_{i=1}^m f_i \right\|_{\mathcal{M}^p_\phi}
	\leq \prod\limits_{i=1}^m \| f_i \|_{\mathcal{M}^{p_i}_{\phi_i}}$
	for every $f_i \in \mathcal{M}_{\phi_i}^{p_i}(\mathbb{R}^d)$, where $i=1,\dots,m$.
		
	{\rm (2)} $\left\| \prod\limits_{i=1}^m f_i \right\|_{w\mathcal{M}^p_\phi}
	\leq m \prod\limits_{i=1}^m \| f_i \|_{w\mathcal{M}^{p_i}_{\phi_i}}$
	for every $f_i \in w\mathcal{M}_{\phi_i}^{p_i}(\mathbb{R}^d)$, where $i=1,\dots,m$.
	\par}
\end{corollary}

\medskip

\textbf{Acknowledgement}. The research is supported by ITB Research \& Innovation Program 2017.
The authors thank the referee for his/her useful remarks on the earlier version of this paper.


\medskip

\begin{thebibliography}{10}
\bibitem{Avram}
F. Avram and L. Brown, ``A generalized H\"{o}lder's inequality and a generalized Szego theorem'',
\emph{Proc. Amer. Math. Soc.} \textbf{107}-3 (1989), 687--695.

\bibitem{Cheung}
W.S. Cheung, ``Generalizations of H\"{o}lder's inequality'', \emph{Int. J. Math. Math. Sci.}
\textbf{26}-1 (2001), 7--10.

\bibitem{Eridani}
Eridani; H. Gunawan; E. Nakai; and Y. Sawano, ``Characterizations for the generalized fractional
integral operators on Morrey spaces'', \emph{Math. Ineq. Appl.} {\bf 17} (2014), 761--767.

\bibitem{EGU12}
Eridani; H. Gunawan; and M.I. Utoyo, ``A characterization for fractional integral operators on
generalized Morrey spaces'', \emph{Anal. Theory Appl.} \textbf{28}-3 (2012), 263--267.

\bibitem{Gunawan}
H. Gunawan; D.I. Hakim; K. Limanta; and A.A. Masta, ``Inclusion properties of generalized Morrey
spaces'', \emph{Math. Nachr.} \textbf{290} (2017), 332--340.

\bibitem{Masta2}
A.A. Masta; H. Gunawan; and W. Setya-Budhi, ``An inclusion property of Orlicz-Morrey spaces'',
\emph{J. Phys.: Conf. Ser.} \textbf{893} (2017), 1--8.

\bibitem{Matkowski}
J. Matkowski, ``The converse of the H\"{o}lder's inequality and its generalizations'',
\emph{Stud. Math.} \textbf{109}-2 (1994), 171--182.

\bibitem{Indra}
I. Rukmana, \textit{Multiplication Operators on Morrey Spaces and Weak Morrey Spaces}
(in Indonesian), Master Thesis, Institut Teknologi Bandung, Bandung, 2016.

\bibitem{SST}
Y. Sawano; S. Sugano; and H. Tanaka, ``Generalized fractional integral operators and
fractional maximal operators in the framework of Morrey spaces'', \emph{Trans. Amer. Math. Soc.}
\textbf{363}-12 (2011), 6481--6503.

\bibitem{Sugano}
S. Sugano, ``Some inequalities for generalized fractional integral operators on generalized
Morrey spaces'', \emph{Math. Ineq. Appl.} \textbf{14}-4 (2011), 849--865.

\bibitem{Sugano1}
S. Sugano, ``Some inequalities for generalized fractional integral operators on generalized
Morrey spaces and their remarks'', \emph{Sci. Math. Japon.} \textbf{76}-3 (2013), 471--495.

\bibitem{Vasyunin}
V. Vasyunin, ``The sharp constant in the reverse H\"{o}lder's inequality for Muckenhoupt weights'',
\emph{St. Petersburg Math. J.} \textbf{15}-1 (2004), 49--79.

\end{thebibliography}
\end{document}